\documentclass [reqno,10pt,oneside]{amsart}

\usepackage{amsthm}
\usepackage{amscd,amssymb}
\usepackage[arrow,matrix]{xy}

\usepackage{graphicx}
\usepackage[noend]{algorithmic}

\oddsidemargin=0.3in \evensidemargin=0.3in \theoremstyle{plain}

\theoremstyle {plain}
\newtheorem {thm}{Theorem}[section]
\newtheorem {proposition}[thm]{Proposition}
\newtheorem {lemma}[thm]{Lemma}

\theoremstyle {definition}
\newtheorem {definition}[thm]{Definition}
\newtheorem {Algorithm}[thm]{Algorithm}
\newtheorem {remark}[thm]{Remark}

\newtheorem {ex}[thm]{Example}

\DeclareMathOperator{\Ann}{Ann}
\DeclareMathOperator{\Ass}{Ass}
\DeclareMathOperator{\minAss}{minAss}
\DeclareMathOperator{\Ext}{Ext}
\DeclareMathOperator{\Supp}{Supp}
\DeclareMathOperator{\codim}{codim}

\newcommand{\Q}{{\mathbb Q}}

\newcommand{\gen}[1]{\left\langle #1 \right\rangle}

\begin{document}

\title{On primary decomposition of modules}

\author{Nazeran Idrees$^1$}
\address{Nazeran Idrees\\ Department of Mathematics\\ GC University\\
Faisalabad\\ Pakistan}
\email{nazeranjawwad@gmail.com}

\author{Afshan Sadiq$^2$}
\address{Afshan Sadiq, \\
Department of Mathematics, \\
Jazan University, \\
P.O. Box 114, Jazan, Saudia Arabia.}
\email{afshansadiq6@gmail.com}

\author{Asifa Tassaddiq$^3$}
\address{Asifa Tassaddiq\\ Department of Mathematics\\ GC University\\
Faisalabad\\ Pakistan}
\email{asifashabbir@gmail.com}


\date{\today}

\maketitle
\begin{abstract}
Primary decomposition is a very important tool of commutative algebra and geometry. In this paper we generalized some of the existing algorithms of primary decomposition developed by Eisenbud et al. (cf. [EHV]) for free modules and also filled some gaps by providing proofs of important Theorems (\ref{thmprimary}, \ref{thmprimmon}, \ref{thmprimcomp}) appeared in [EHV]. All these algorithms are programmed and implemented in {\sc Singular}.
\end{abstract}

\section{Introduction }

Primary decomposition is a very well known and active area of research in computational algebra. Most of the algorithms for primary decomposition depend on Gr\"obner basis. The idea of Gianni, Trager and Zacharias (cf. [GTZ]) depend on reducing the ideal to zero dimensional case and applying linear coordinate change, and these techniques are generalised by E. W. Rutman (cf. [R]) for modules. The algorithm of Shimoyama and Yokoyama (cf. [SY]) introduces the concepts of pseudo primary decomposition and separating sets and these methods are further enhanced by Noro (cf. [N]) and are generalised for modules by Idrees (cf. [I]) . Many further developements are made in these main existing algorithms. Modular and parallelization techniques for these algorithms are discussed and implemented in SINGULAR by N. Idrees, G. Pfister and S. Steidel (cf. [IPS]). The idea of primary decomposition of Eisenbud, Huneke and Vasconcelos (cf. [EHV]) employs the methods of equidimensional decomposition of ideals and homological methods which are here generalised for modules partially and provided proofs of important theorems, these algorithms are implemented in SINGULAR (cf. DGPS ).
We will assume that all modules are finitely generated.

\begin{definition}
 {\em Let $N\subset M$ be
submodules of $R^s=\mathbb{Q}[X]^s$. We say that $N$ is a {\em
primary submodule\em} of $M$ if for $r\in R,v\in M$ and $rv\in N
\Rightarrow v\in N$ or $r\in\sqrt{\Ann(M/N)}$. In this case
$\Ann(M/N)$ is a primary ideal of R, and we say that N is {\em
$\sqrt{\Ann(M/N)}$-primary\em} in M.\em}
\end{definition}
 \begin{remark}
  Let $N\subset M$ be
$R$-modules. If $\emph{P}$ is a maximal ideal of $\emph{R}$ then $N$
is $\emph{P}$-primary in $M\,\Leftrightarrow \Ann(M/N)$ is a
$\emph{P}$-primary ideal in $R$.
\end{remark}
\begin{definition}
{\em A submodule $N $ of an $R$-module $M\subset R^s$
has a {\em primary decomposition\em} if $N=\cap_{i=1}^r Q_i$ with
each $Q_i$ a $P_i $-primary submodule of $M$ for some prime ideal
$P_i$ of R. $Q_i $ is called the {\em primary component\em} of $N$
belonging to $P_i$ and each $P_i$ is an {\em associated prime\em} of
$N$. If $Q_i$ does not contain $\cap_{j\neq i}Q_j$ and the $P_i$ are
all distinct then the decomposition is said to be {\em reduced\em}. As intersection of a finite number of $P$-primary submodules of M is
also $P$-primary, so we can always have a reduced primary
decomposition from a given primary decomposition. If $P_i $ is a
minimal prime ideal among the set of all associated primes of $N$,
then $P_i$ is called {\em isolated\em} prime associated to $N$;
otherwise $P_i$ is {\em embedded\em}. The set of all minimal
associated primes of $N$ is called $\minAss(N)$ and the set of all
associated primes is called $\Ass(N)$\em}.
\end{definition}
\begin{definition}
 The equidimensional hull of $0$
 in a module $M$ is defined to be the submodule $N$
 which contains all elements whose annihilators have dimension strictly less than the dimension of $M$. In other words $N$ is the intersection of all primary components of $0$ in $M$ having maximal dimension. We define the equidimensional hull of a submodule $M'\subset M$ to be the preimage in $M$ of the equidimensional hull of $0$ in $M/M'$.
\end{definition}
We write hull($N,M$) for the equidimensional hull , or simply by hull$N$, when there is no danger of confusion.
\section{Proofs and Algorithms}
\begin{thm}
Let $M$ be a module over regular domain $S$, and set $I_c=\Ann{\Ext_S}^c(M,S)$:
\begin{itemize}
 \item[1.] $I_c$ has codimension greater than or equal to $c$ and $M/(0:_MI_c)$ has no associated primes of codimension $c$. In particular a prime ideal of codimension $c$ is associated to $M$ if and only if $P$ contains the annihilator of ${\Ext_c}^S(M,S)$.
\item[2.] The equidimensional hull of $0$ in $M$ is the kernel of the natural map $$\pi: M\rightarrow {\Ext_S}^c({\Ext_S}^c(M,S),S)$$ where $c$ is the codimension of $M$.
    \end{itemize}
\end{thm}
This theorem can be used to find the equidimensional hull of a module and also to remove the components of dimension less than any given number.
\\We give algorithms to remove component of low dimension and to find equidimensional hull of given module $M$.

\begin{Algorithm}
\textsc{RemComp(M)}\label{algRemCom}
\begin{algorithmic}
\REQUIRE Given a module $M$ over $S=K[x_1,x_2,...,x_n]$, and an integer $c$ (normally taken $\geq\dim M$).
\ENSURE A submodule $N_c$ consisting of the intersection of the primary components of $M$ of dimension greater than or equal to $c$.
\vspace{0.1cm}
\STATE $b:=\dim S$;
\STATE $N:=0$;
\WHILE{$b>c$}
\STATE compute $\Ext^b(M,S)$;
\IF{codim $\Ext^b(M,S)=b$}
\STATE $I_b:=\Ann(\Ext^b(M,S))$;
\STATE $N:=(N:_MI_b)$;
\STATE decrement $b$;
\STATE(Optional: $m:=M/N$);
\ENDIF
\ENDWHILE
\RETURN $N$;
\end{algorithmic}
\end{Algorithm}
\begin{Algorithm}
\textsc{EquidimHull(M)}\label{algEquidim}
\begin{algorithmic}
\REQUIRE Given a finitely generated module $M$ over $S=K[x_1,x_2,\ldots x_n]$.
\ENSURE The equidimensional  kernel $N\subset M$.
\vspace{0.1cm}
\STATE $c:=\codim M$;
\STATE compute $N := \ker (M\rightarrow {\Ext_S}^c({Ext_S}^c(M,S),S))$, the kernel of the canonical map;
\RETURN $N$;
\end{algorithmic}
\end{Algorithm}
In practice the canonical map is computed by forming the comparison map between the dual of a free resolution of $M$ and a free resolution of $\Ext^c_M(S/I,S)$.
An alternative would be to construct a polynomial subring $T$ of $S$ such that $\dim T = \dim N$ and over which $N$ is finitely generated (a Noether normalization for $S/\Ann(N)$ will do) and then take the kernel of the natural map of $N$ into its double dual over $T$.
The following algorithm will be very useful for the purposes of localization:
\begin{Algorithm}
\textsc{AssPrimCodimc(M)}\label{algAssprimcodim}
\begin{algorithmic}
\REQUIRE Given a finitely generated module $M$ over $S=K[x_1,x_2,\ldots,x_n]$ and an integer $c$.
\ENSURE Find an ideal whose associated primes are exactly the associated primes of $M$ having codimension $c$.
\vspace{0.1cm}
\STATE $I_c:=\Ann \Ext_S^c(M,S)$;
\IF{$(\codim I_c > c$)}
\STATE $I=S$;
\ELSE
\STATE $I$ := \textsc{EquidimHull}$(I_c)$;
\ENDIF
\RETURN $I$;
\end{algorithmic}
\end{Algorithm}
\begin{lemma} \label{lemPrim}
If $J$ is an ideal in noetherian ring $R$ and $M$ is a finitely generated $R$-module, then
\begin{itemize}
\item[1.] $(0:_M J^\infty)$ is the intersection of primary components of $0$ in $M$ whose associated primes do not contain $J$.
\item[2.] $\minAss M \cap \Supp JM \subset \Ass(0:_J M)\subset \Ass(M)\cap \Supp JM $
where $\minAss M$ is the set of minimal associated primes of $M$. Further given a primary decomposition of $0$ in $M$,
there is a primary decomposition of $(0:_MJ)$ for which each primary component contains the corresponding primary component
of $0$.
\item[3.] In particular if $I$ is a radical ideal then $(I:J)$ is radical and
      $$ (I:J)=\bigcap P_j$$
   where $P_j$ ranges over all primes containing $I$ but not containing $J$.
\end{itemize}
\end{lemma}
Here is an algorithm for finding the intersection of associated primes of module $M$ having a given dimension $c$.
\begin{Algorithm}
\textsc{InterAssPrim(M)}\label{algAssprimcodim}
\begin{algorithmic}
\REQUIRE A module $M$ and an integer $c$.
\ENSURE A set $H$ which is intersection of all associated primes of $M$  having dimension $c$.
\vspace{0.1cm}
\STATE $I_c:=\Ann \Ext_S^c(M,S)$;
\IF{$\codim I_c = c$}
\STATE $H$:= radical of eqidimensional hull of $I_c$;
\ELSE
\STATE $H := S$;
\ENDIF
\RETURN $H$;
\end{algorithmic}
\end{Algorithm}
\begin{definition}
Let $S$ be an affine ring and $J\subset S$ be an ideal, and $A\subset B$ be a finitely generated $S$ modules.
The localization of $A$ at the ideal $J$ is, denoted by $A_J$, is defined as
\\$A_J=\{b\in B\mid \dim(J+(A:b)) < \dim J\}$.
If $J$ is a prime ideal then we can write as
\\$A_J=\{b\in B\mid (A:b)\subsetneq J\}$.
If $J$ is a prime ideal , $A_J$ is the preimage of the usual localization $A_j\subset B_J$
under the canonical map $B\rightarrow B_J$. \\Of course the localization at $J$ is the same as the localization at the equidimensional hull of $J$.
\end{definition}
The following proposition is helpful in actual computation of localization of a module at an ideal.

\begin{proposition}
Let $J\subset S $ be an ideal in a noetherian ring, let $A\subset B$ be S-modules,
 and let $I_c'$ be the intersection of all the associated primes of $B/A$ having codimension exactly $c$. If we set $$K:=\bigcap_c(I_c',(I_c')_J)$$
 then, $$A_{[J]}=(A:K^{\infty}).$$
\end{proposition}
 \begin{proof}
By Lemma \ref{lemPrim} (2) the ideal $K_c:=(I_c',(I_c')_J)$ is the  intersection of those associated primes of $B/A$ having codimension $c$ and not contained in a prime containing $J$ and having the same dimension as $J$.
By Lemma \ref{lemPrim} (3) $(A:K^{\infty})$ is the result of removing all corresponding primary components from $A$, and is equal to $A_{[J]}$.
\end{proof}
\begin{Algorithm}
\textsc{Local(A,j)}\label{alglocal}
\begin{algorithmic}
\REQUIRE A module $A\subseteq B$ and an ideal $J$.
\ENSURE A module $A_J=\{b\in B\mid (A:b)\subsetneq J\}$, localization of $A$ at $J$.
\vspace{0.1cm}
\FOR{for each $c = \codim B/A,\ldots, n$}
\STATE compute $I_c'$ = intesection of all associated primes  of $B/A$ having codimension $e$;
\STATE $I_c'' := {I_c'}_{(J)}$;
\STATE $I_c''' := (I_c':I_c")$;
\STATE $K := \bigcap_c I_c'''$;
\ENDFOR
\RETURN $A_{[J]} := (A:K^{\infty})$
\end{algorithmic}
\end{Algorithm}

\begin{thm}\label{thmprimary}
Let $M=Q_1\cap Q_2\cap\ldots Q_m$ be an irredundant primary decomposition, $P_i=\sqrt{Q_i}$.
Let $P\in \{P_1,P_2,\ldots,P_m\}$ and $Q$ be a $P$-primary module such that $M\subseteq Q$. Then $Q$ is a primary component for $M$, i.e.
there exists $i$ such that $M=Q_1\cap\ldots Q_{i_1}\cap Q\cap Q_{i+1}\ldots\cap  Q_m$ if and
only if $Q\cap (M_{[P]}:P^\infty)=M_{[P]}$.
\end{thm}
\begin{proof}Assume $Q=Q_i$ for some $i$. Then $P=P_i$.
\\Claim. If $P$ is not embedded then $M_{[P]}=Q$ and $M_{[P]}:P^\infty=R$.
\\Obviously $M_P\subseteq Q$. Let $x\in Q$ then $$M:x=(Q_1:x) \cap \ldots (Q_{i-1}:x)\cap (Q:x)\cap (Q_{i+1}:x)\ldots\cap (Q_m:x).$$
If $(M:x)\subseteq P=P_i$, this implies that $(Q_j:x)\subseteq P_i$ for some $j\neq i$.
This implies that $P_j\subseteq P_i$ which is a contradiction to the assumption that $P=P_i$ is not embedded.
This implies $I:x\varsubsetneq P$ and therefore $x\in M_{[P]}$. This implies $M_{[P]}=Q$. Now $P^m\subseteq Q$ for some $m$.
This implies that $M_{[P]}:P^\infty=R$ and proves the claim.
\\Claim.If $P_j\subsetneqq P$ then $M_{[P]}:P^\infty\subseteq M_{[P_j]}\subseteq Q_j$.
\\Let $x\in M_{[P]}:P^\infty$ i.e. $xq\in M_{[P]}$ for all $q\in P^m$ for all $m$.
There exists $\xi\notin P,\xi qx\in M$.
now choose $q\in P\setminus P_j$ then $\xi q\notin P_j$, this implies $x\in M_{P_j}$ and proves the claim.
Now let $x\in Q\cap(I_{[P]}:P^\infty)$ then $$M:x=(Q_1:x)\cap (Q_2:x)\cap \ldots \cap (Q_{i-1}:x)\cap (Q_{i+1}:x)\ldots$$
But $x\in M_{[P]}:P^\infty\subseteq Q_j$ for all $P_j\subsetneq P$. This implies $M:x=\cap_{P_j\subsetneq P}Q_j:x$.
 If $M:x\subset P$ then $Q_j:x\subset P$ for some $j$ with $P_j\subsetneq P$. But this is not possible.
 This implies $M:x\subsetneq P$, i.e. $x\in M_{[P]}$. Therefore $Q\cap (M_{[P]}:P^\infty)\subseteq M_{[P]}$.
 The other inclusion is obvious.
 We proved one direction of theorem.
\\ To prove the other direction assume $P=P_i$ and $Q\cap (M_{[P]}:P^\infty)=M_{[P]}$.
\\  Therefore $Q\cap (M_{[P]}:P^\infty)$=$Q_i\cap (M_{[P]}:P^\infty)$.
  Let $$N=Q_1\cap\ldots Q_{i_1}\cap Q_{i-1}\cap Q_{i+1}\ldots\cap Q_m.$$
  We have to prove that $M=N\cap Q$, i.e. $N\cap Q=N\cap Q_i$.
  \\Claim $N=M_{[P]}:P^\infty)\cap (\cap_{P_i\subsetneq P}Q_i)$.
\\Let $x\in N$, choose $m$ such that $P^m\subset Q_i$ then $xP^m\subseteq M\subseteq M_{[P]}$.This proves that $N\subseteq M_{[P]}:P^\infty$.
On the other hand we proved already that $M_{[P]}:P^\infty)\subseteq (\cap_{P_i\subsetneq P}Q_i)$.
 This implies that $M_{[P]}:P^\infty)\cap (\cap_{P_i\subsetneq P}Q_i)\subset N$
  but $N\subseteq M_{[P]}:P^\infty)$ implies $N=M_{[P]}:P^\infty)\cap (\cap_{P_i\subsetneq P}Q_i)$ and proves the claim.
   As $Q\cap (M_{[P]}:P^\infty)$=$Q_i\cap (M_{[P]}:P^\infty)$ so we obtain $N\cap Q=N\cap Q_i$.
\end{proof}
\begin{ex}
Let $I=\langle x^2,xy\rangle \subset \Q[x,y]$ be an ideal, then $P=\langle x\rangle,\,I_{[P]}=I$, $I_{[P]}:P^\infty=\gen x,\,Q = \langle x^2,y\rangle$ and $Q_1 = \langle x^2,xy,y^2\rangle$.
\end{ex}
Now we are ready to give an Algorithm for finding a primary component for a given associated prime.

\begin{Algorithm}
\textsc{PrimComp(A,P)}\label{Algprimcomp}
\begin{algorithmic}
\REQUIRE A module $A\subseteq B$ and a prime ideal $P$.
\ENSURE A primary component $Q$ of $A$ with associated prime $P$.
\vspace{0.1cm}
\STATE $T := PB$;
\STATE compute $A_{[P]}$ := \textsc{Local}($A,P$);
\STATE compute $A_{[P]}:P^{\infty}$;
\STATE compute \textsc{EquidimHull}($A+T$);
\IF{$A_{[P]}:P^{\infty} \subset A_{[P]}$}
\STATE $Q$ := \textsc{EquidimHull}$(A+T)$;
\ELSE
\STATE $T := PT$;
\ENDIF
\RETURN $Q$;
\end{algorithmic}

\end{Algorithm}

\begin{thm}\label{thmprimmon}
Let $Q$ be a submodule of $R[x]^s$ and $I$ be an ideal in $R$ then the natural map $Q_{[P]}:P^\infty /Q_{[P]}\rightarrow R/Q$, where $P$ is minimal
 associated prime of $I$, is a monomorphism if and only if $Q$ is $P$-primary.
\end{thm}
\begin{proof} Suppose $Q$ is $P$ primary. First we will show that $M_{[P]}:P^\infty \cap Q\subset M_{[P]}$.
Note that $M_{[P]}=\{b\in R^S\mid dim(P+M:b)<dim\, P\} = \{b\in R^S\mid M:b\nsubseteq P\}\subseteq P$
because if $\xi \in M:b$ i.e. $\xi b\in M$ and $\xi \notin P $ implies $b\in P$.
So $M_{[P]}=\{b\in P\mid M:b\nsubseteq P\}\subseteq Q$.\\Choose $\xi \notin P$ but $b\xi \in M$ which implies $b\in Q$.
To prove that $M_{[P]}:P^\infty \cap Q=M_{[P]}$, it is enough to prove that for any $p\in P$, $M_{[p]}:p \cap Q=M_{[p]}$.
\\Let $f\in M_{[P]}:P^\infty \cap Q$. There exists $\xi$ such that $p.f.\xi\in M, \,\,\xi\notin P,\,\, f\in Q,\,\,p\in P$. To show $f\in M_{[P]}$ we have to show that $M:f\nsubseteq P$.\\
$M = Q_1\cap Q_2\cap \ldots \cap Q_s$ so $M:f = (Q_1:f)\cap (Q_2:f) \cap \ldots \cap (Q_s : f)$. Assume $M:f\subseteq P$, which implies there exists
$i$ such that $Q_i\subsetneq Q_i:f\subseteq P$. This implies $P_i\subsetneq P$, which is contrary to the fact that $P$ is minimal.\\
Conversely, we suppose that $M_{[P]}:P^\infty /M_{[P]}\rightarrow R[x]^s/\overline{Q}$ is injective and $\overline{Q}$ is $P$-primary module containing $M$,
i.e $\overline{Q}=M_{[P]}=\{b\mid M:b\nsubseteq P\}\subseteq Q$ (we know $\overline{Q}$ is $P$-primary and $P$ is associated prime of $M$ and $M$ is equidimensional).
$M= Q_1\cap Q_2\cap\ldots Q_s$, this implies that $N\subseteq\overline{Q}\subseteq Q$, which in turn implies that $M=\overline{Q}\cap Q_2\ldots Q_s$, so 
$Q=\overline{Q}$.
\end{proof}
\begin{thm}\label{thmprimcomp}
Let $N\subseteq R[x]^s$ be a module, $P\in Ass(M)$. Then $N+P^mR^S$ is a $P$-primary component of $N$ for some $m$.
\end{thm}
\begin{proof}
Let $N= Q_1\cap Q_2\ldots\cap Q_s\cap Q_{s+1}\cap \ldots \cap Q_m$, $P=\sqrt Q, \,Q=Q_i$. Choose $m$ such that $P^m\in Q$.
Now $N + P^m R^S\subseteq \cap_{P\subseteq \sqrt{Q_j}}Q_j$, $f\in N+P^mR^S$, so $f$ is of the form $f = f_1 + f_2$, this implies $f_1\in Q_j$ for any $j$,
$f_2\in Q_j,\,\, P\in\sqrt{Q_j}$, so $f\in \cap_{P\subseteq \sqrt{Q_j}}Q_j$, (as equi($\cap_{P\subseteq \sqrt{Q_j}}Q_j)=Q$) so equi$(N+P^m)=Q$.
\end{proof}
Here is the Algorithm to find the primary decomposition of a given module.

\begin{Algorithm}
\textsc{PrimdecmEHV(M)}\label{primdecehv}
\begin{algorithmic}
\REQUIRE A module $M\subset R[x]^n$, and $x={x_1,x_2,\ldots,x_n}$.
\ENSURE A list $(Q_i,P_i)$, where $Q_i$ is primary component of $A$, with associated prime $P_i$.
\vspace{0.1cm}
\STATE $N :=$ \textsc{EquidimHull}$(M)$;
\STATE compute $B:=\{P_1, \ldots, P_r\}$, the set of minimal associated primes of $\Ann(N)$;
\FOR{$i = 1, \ldots, r$}
\STATE compute $Q_i =$ \textsc{PrimComp}$(N,P_i)$;
\ENDFOR
\IF{$M$ has embedded primes}
\WHILE{$(f > \codim(M))$}
\STATE $H:={\Ext_R}^f (M)$;
\STATE $I_f := \Ann(H)$
\STATE $c := \codim(I_f)$;
\IF{$(c=f)$}
\STATE $K := \minAss($ \textsc{EquidimMaxEHV}($I_f$));
\STATE compute for each prime ideal $P_i$ in $K$;
\STATE $Q_i =$ \textsc{PrimComp}$(M,P)$;
\STATE $f = f-1$;
\ENDIF
\ENDWHILE
\ENDIF
\RETURN $(P_i,Q_i)$;
\end{algorithmic}
\end{Algorithm}

 In programming, we computed associated primes using the method of (cf. [GTZ]) instead of eisenbud et al. It would be interesting to compare the results with computation of associated primes using other techniques.

\section{Procedures}

      \texttt{proc primEHV(module M)}
 \\  \texttt{"USAGE:  primEHV(id); id= ideal/module,}
 \\     \texttt{RETURN: a list K of primary ideals and their associated primes:}
  \\     \texttt{ @*   K[i][1]   the i-th primary component of M,}
  \\   \texttt{@*   K[i][2]   the i-th prime component of M.}
  \\ \texttt{EXAMPLE: example primEHV; shows an example}
\\ \texttt{"}
\\ $\{$
   \\  \indent \quad   \texttt{list Z,L,K,W;}
   \\  \indent \quad   \texttt{module H;}
   \\  \indent \quad   \texttt{ideal If;}
   \\  \indent \quad   \texttt{int i,e,n,c;}
   \\  \indent \quad   \texttt{n=nvars(basering);}
   \\  \indent \quad   \texttt{e=dim(std(M));}
   \\  \indent \quad   \texttt{int f=n;}
  \\  \indent \quad   \texttt{module M1=canonMap(M)[1];}
 \\    \indent \quad   \texttt{module N1=freemodule(nrows(M));}
   \\  \indent \quad   \texttt{module N=N1;}
   \\  \indent \quad   \texttt{L=minAssGTZ(Ann(M1));}
      \\  \indent \quad   \texttt{int l = size(L);}
   \\  \indent \quad   \texttt{for( i=1; i<=l; i++)}
   \\  \indent \quad  $\{$
      \\  \indent \quad \quad   \texttt{K[i] = list();}
     \\  \indent \quad \quad   \texttt{ K[i][2] =std(L[i]);}
     \\  \indent \quad  \quad  \texttt{ K[i][1] = com(M1,N1,std(L[i]),L);}
   \\  \indent \quad  $\}$
   \\  \indent \quad   \texttt{for(i=1;i<=size(K);i++)}
   \\  \indent \quad  $\{$
   \\  \indent \quad   \quad   \texttt{   N=intersect(N,K[i][1]);}
    \\  \indent \quad  $\}$
   \\  \indent \quad   \texttt{if(reduce(N,std(M))!=0)  }       //if M has embedded primes then
  \\  \indent \quad $\{$
    \\  \indent \quad \quad   \texttt{ while(f>n-e)}
     \\  \indent \quad \quad $\{$
     \indent \quad   \quad \quad \texttt{H=ExtR(f,M);}
    \\  \indent \quad \quad \quad   \texttt{   If=quotient(H,freemodule(nrows(H)));}      //If is ann(H)
    \\  \indent \quad \quad \quad   \texttt{   c=n-dim(std(If));}
       \\  \indent \quad  \quad \quad \texttt{if(c==f)}
       \\  \indent \quad  \quad \quad $\{$
       \\  \indent \quad  \quad \quad \quad  \texttt{  Z=minAssGTZ(equiMaxEHV(If));}
       \\  \indent \quad  \quad \quad \quad  \texttt{  for( i=1; i<=size(Z); i++)}
         \\  \indent \quad \quad \quad \quad  $\{$
       \\  \indent \quad   \quad \quad \quad \quad \texttt{    W[i] = list();}
       \\  \indent \quad   \quad \quad \quad \quad  \texttt{    W[i][2] =std(Z[i]);}
       \\  \indent \quad \quad \quad \quad \quad   \texttt{    W[i][1] = com(M,N1,std(Z[i]),Z);}
       \\  \indent \quad \quad \quad \quad \quad   \texttt{    N=intersect(N,W[i][1]);}
          \\  \indent  \quad \quad \quad\quad  $\}$
       \\  \indent \quad  \quad \quad \quad  \texttt{  K=K+W;}
        \\  \indent  \quad\quad \quad  $\}$
       \\  \indent \quad  \quad \quad  \texttt{f--;}
    \\  \indent \quad   \quad$\}$
  \\  \indent \quad  $\}$
 \\  \indent \quad   \texttt{return(K);}
  \\  $\}$

   \texttt{example}
 \\ $\{$  \indent \quad   \texttt{   "EXAMPLE:"; echo = 2;}
 \\  \indent \quad   \texttt{     ring s=0,(x,y,z),dp;}
 \\  \indent \quad   \texttt{      ideal i=x2y,xz2,y2z;}
 \\  \indent \quad   \texttt{     primEHV(i);}
 \\  \indent \quad   \texttt{   ring T = 0,(x,y,z),dp;}
 \\  \indent \quad   \texttt{    module M=[xy,0,yz],[0,xz,z2];}
 \\  \indent \quad   \texttt{     primEHV(M);}
\\$\}$

    \texttt{  proc canonMap(list l)}
 \\   \texttt{  "USAGE:  canonMap(id); id= ideal/module,}
 \\ \texttt{  RETURN:  a list L, the kernel in two different representations and}
\\    \texttt{  @*       the cokernel of the canonical map}
\\  \indent \quad   \texttt{  @*       M ---> Ext\^c\_R(Ext\^c\_R(M,R),R) given by presentations}
 \\    \texttt{  @*       Here M is the R-module (R=basering) given by the }
 \\   \texttt{  @*      presentation defined by id, i.e. M=R/id resp. M=R\^n/id}
 \\   \texttt{  @*       c is the codimension of M  }
 \\   \texttt{  @*       L[1] is the preimage of the kernel in R resp. R\^n}
 \\     \texttt{  @*       L[2] is a presentation of the kernel}
 \\   \texttt{  @*       L[3] is a presentation of the cokernel}
  \\ \texttt{EXAMPLE: example canonMap; shows an example}
  \\ \texttt{"}\\
$\{$
 \\  \indent \quad   \texttt{  module M=hash[1];}
 \\  \indent \quad   \texttt{  int c=nvars(basering)-dim(std(M));}
 \\  \indent \quad   \texttt{  if(c==0)}
   \\  \indent \quad  $\{$
 \\  \indent \quad \quad   \texttt{     module K=syz(transpose(M));}
 \\  \indent \quad \quad   \texttt{     module Ke=syz(transpose(K));}
 \\  \indent \quad  \quad  \texttt{     module Co=modulo(syz(transpose(syz(K))),transpose(K));}
    \\  \indent \quad  $\}$
 \\  \indent \quad   \texttt{  else}
   \\  \indent \quad  $\{$
 \\  \indent \quad  \quad  \texttt{     int i;}
 \\  \indent \quad  \quad  \texttt{     resolution F=mres(M,c+1);}
 \\  \indent \quad  \quad  \texttt{     module K=syz(transpose(F[c+1]));}
 \\  \indent \quad  \quad  \texttt{     K=simplify(reduce(K,std(transpose(F[c]))),2);}
 \\  \indent \quad  \quad  \texttt{     module A=modulo(K,transpose(F[c]));}
 \\  \indent \quad  \quad  \texttt{     resolution G=nres(A,c+1);}
 \\  \indent \quad   \quad \texttt{     for(i=1;i<=c;i++)}
      \\  \indent \quad \quad  $\{$
 \\  \indent \quad  \quad \quad  \texttt{        K=lift(transpose(F[c-i+1]),K*G[i]);}
       \\  \indent \quad\quad  $\}$
 \\  \indent \quad  \quad  \texttt{     module Ke=modulo(transpose(K),transpose(G[c]));}
 \\  \indent \quad  \quad  \texttt{     module Co=modulo(syz(transpose(G[c+1])),transpose(K)+transpose(G[c]));}
    \\  \indent \quad  $\}$
 \\  \indent \quad   \texttt{  return(list(Ke,Co));}
 \\   $\}$

   \texttt{ example}
\\$\{$   \indent \quad   \texttt{  "EXAMPLE:"; echo = 2;}
    \\  \indent \quad   \texttt{  ring s=0,(x,y),dp;}
   \\  \indent \quad   \texttt{   ideal i = x,y;}
   \\  \indent \quad   \texttt{   canonMap(i);}
  \\  \indent \quad   \texttt{    ring R = 0,(x,y,z,w),dp;}
  \\  \indent \quad   \texttt{    ideal I1 = x,y;}
  \\  \indent \quad   \texttt{    ideal I2 = z,w;}
  \\  \indent \quad   \texttt{    ideal I = intersect(I1,I2);}
  \\  \indent \quad   \texttt{    canonMap(I);}
  \\  \indent \quad   \texttt{    module M = syz(I);}
  \\  \indent \quad   \texttt{    canonMap(M);}
  \\  \indent \quad   \texttt{    ring S = 0,(x,y,z,t),Wp(3,4,5,1);}
  \\  \indent \quad   \texttt{    ideal I = x-t3,y-t4,z-t5;}
  \\  \indent \quad   \texttt{    ideal J = eliminate(I,t);}
  \\  \indent \quad   \texttt{    ring T = 0,(x,y,z),Wp(3,4,5);}
  \\  \indent \quad   \texttt{ideal p = imap(S,J);}
  \\  \indent \quad   \texttt{ideal p2 = p\^2;}
  \\  \indent \quad   \texttt{canonMap(p2);}
\\$\}$

     \texttt{  proc com(module A, module B, ideal P, list L)}
  \\    \texttt{  "USAGE:com(id1,id2,P,L);id1=ideal/module,id2=ideal/module ,P prime}
\\ \texttt{  @* ideal in the list L of prime ideals}
 \\     \texttt{  RETURN: returns a primary component of the module A}
 \\     \texttt{             @* defined by id1 associated }
   \\   \texttt{ @*to prime ideal P defined by id2  }
  \\ \texttt{EXAMPLE: example com; shows an example}
  \\ \texttt{"}\\

 $\{$
 \\  \indent \quad   \texttt{ module T = P*B;}
 \\  \indent \quad   \texttt{ module Q;}
 \\  \indent \quad   \texttt{ module AP = groebner(locm(A,B,P,L));}
\\  \indent \quad  $\{$
  //...and compute the saturation of the localization w.r.t. P.
 \\  \indent \quad \quad   \texttt{ module AP2 = sat(AP,P)[1];}

  //As long as we have not found a primary component...
 \\  \indent \quad \quad   \texttt{ int isPrimaryComponent = 0;}
 \\  \indent \quad \quad   \texttt{ while(isPrimaryComponent!=1)}
   \\  \indent \quad  \quad $\{$
      //...compute the equidimensional part Q of A+P\^n...
       \\  \indent \quad \quad \quad   \texttt{Q = canonMap(A+T)[1];}
     //and check if it is a primary component for P.
 \\  \indent \quad   \quad \quad \texttt{if(isSub(intersect(AP2,Q),AP)==1)}
      \\  \indent  \quad \quad\quad  $\{$
 \\  \indent \quad   \quad \quad \quad \texttt{isPrimaryComponent = 1;}
       \\  \indent \quad \quad \quad  $\}$
 \\  \indent \quad    \quad \quad\texttt{else}
      \\  \indent  \quad \quad\quad  $\{$
 \\  \indent \quad \quad \quad \quad   \texttt{T = P*T;}
       \\  \indent \quad \quad \quad  $\}$
    \\  \indent  \quad\quad  $\}$
 \\  \indent \quad   \texttt{  return(Q);}
  \\$\}$

 \texttt{  example}
\\$\{$   \indent \quad   \texttt{  "EXAMPLE:"; echo = 2;}
   \\  \indent \quad   \texttt{   ring r=0,(x,y),dp;}
 \\  \indent \quad   \texttt{  module N=x*gen(1)+ y*gen(2),}
 \\  \indent \quad   \texttt{  x*gen(1)-x2*gen(2);}
 \\  \indent \quad   \texttt{  list L=minAssGTZ(Ann(N));}
 \\  \indent \quad   \texttt{  ideal P=x;}
 \\  \indent \quad   \texttt{  module A=freemodule(nrows(N));}
 \\  \indent \quad   \texttt{  com(N,A,P,L);}
\\$\}$

    \texttt{  proc locm(module A,module B,ideal J,list L)}
 \\   \texttt{  "USAGE: locm(id1,id2,id,list); id1= ideal/module,}
 \\     \texttt{  @*      id2=ideal/module,id=prime ideal in a list L.}
 \\    \texttt{  RETURN:  The localization of a module A denoted by id1}
\\    \texttt{  @*      at the prime ideal J denoted by A\_[J] defined as A\_[J]=(A:K\^(infinity))}
 \\   \texttt{  @*      K is intersection of (I\_e:(I\_e)\_J) over all e, where I\_e is}
 \\   \texttt{  @*      intersection of all associated primes of B/A }
 \\     \texttt{  @*      having codimension e, where A is subset of B,are}
 \\    \texttt{  @*      modules over freemodule S.}
   \\ \texttt{EXAMPLE: example canonMap; shows an example}
  \\ \texttt{"}\\
$\{$
 \\  \indent \quad   \texttt{ ideal h=quotient(A,B);}
 \\  \indent \quad   \texttt{ int n=nvars(basering);}
 \\  \indent \quad   \texttt{ list LL=L;}
 \\  \indent \quad   \texttt{ ideal I,G,P,Q;}
  //assume J is in L
 \\  \indent \quad   \texttt{ int i,c;}
 \\  \indent \quad   \texttt{ list H;}
 \\  \indent \quad   \texttt{ while(size(L)>0)}
\\  \indent \quad  $\{$
 \\  \indent  \quad\quad   \texttt{ I=L[1];}
 \\  \indent  \quad\quad   \texttt{L=delete(L,1);}
 \\  \indent  \quad\quad   \texttt{ c=dim(std(I));}
 \\  \indent  \quad\quad   \texttt{ i=1;}
 \\  \indent  \quad\quad   \texttt{ while(i<=size(L))}
    \\  \indent \quad \quad $\{$
 \\  \indent \quad \quad \quad   \texttt{ if(dim(std(L[i]))==c)}
 \\  \indent \quad    \quad \quad$\{$
 \\  \indent \quad    \quad \quad \quad\texttt{I=intersect(I,L[i]);}
 \\  \indent \quad   \quad \quad \quad \texttt{L=delete(L,i);}
 \\  \indent \quad   \quad \quad \quad \texttt{ i--;}
       \\  \indent  \quad \quad\quad  $\}$
 \\  \indent \quad  \quad \quad  \texttt{i++;	}
     \\  \indent \quad \quad  $\}$
 \\  \indent \quad   \quad \texttt{ H[size(H)+1]=I;}
   \\  \indent \quad  $\}$
 \\  \indent \quad   \texttt{ ideal K=ideal(1);}
 \\  \indent \quad   \texttt{ for(i=1; i<=size(H);i++)}
  \\  \indent \quad  $\{$
 \\  \indent \quad   \quad \texttt{ G=localize(H[i],J,LL);}
 \\  \indent \quad   \quad \texttt{ P=quotient(H[i],G);}
 \\  \indent \quad   \quad \texttt{ K=intersect(K,P);}
   \\  \indent \quad  $\}$
 \\  \indent \quad   \texttt{ return(sat(A,K)[1]);}
 \\  $\}$
 \\    \texttt{  example}
\\$\{$   \indent \quad   \texttt{   "EXAMPLE:"; echo = 2;}
    \\  \indent \quad   \texttt{  ring r=0,(x,y),dp;}
    \\  \indent \quad   \texttt{  module M=[x2,xy2],[xy,y2];}
    \\  \indent \quad   \texttt{  module A=freemodule(nrows(M));}
    \\  \indent \quad   \texttt{  list L=minAssGTZ(Ann(M));}
    \\  \indent \quad   \texttt{  ideal J=L[2]=y-1;}
    \\  \indent \quad   \texttt{  locm(M,A,J,L);}
\\$\}$
 \\    \texttt{  proc isSub(module I,module J)}
 \\     \texttt{  "USAGE:  isSub(mod1,mod2); mod1= ideal/module,}
 \\   \texttt{  @*                  mod2=ideal/module}
 \\     \texttt{  RETURN:  1 if I is a submodule of J else 0.}
    \\ \texttt{"}
\\$\{$
 \\  \indent \quad   \texttt{  int s = size(I);}
 \\  \indent \quad   \texttt{  for(int i=1; i<=s; i++)}
   \\  \indent \quad  $\{$
 \\  \indent \quad \quad   \texttt{     if(reduce(I[i],std(J))!=0)}
      \\  \indent \quad \quad  $\{$
 \\  \indent \quad   \quad \quad \texttt{         return(0);}
       \\  \indent \quad \quad  $\}$
    \\  \indent \quad  $\}$
  \\  \indent \quad   \texttt{ return(1);}
\\ $\}$

\end{document}